\documentclass[12pt,a4paper,twoside,leqno]{amsart}

\usepackage{amsmath}
\usepackage{amsthm}
\usepackage{amsfonts}
\usepackage{amssymb}
\usepackage{amstext}
\usepackage{amsbsy}

\usepackage{epsf}

 \usepackage[active]{srcltx}

\usepackage{xcolor}
\usepackage[colorlinks=true,linkcolor=blue,citecolor=red,filecolor=yellow,
linkcolor=purple,urlcolor=green]{hyperref}

\newtheorem{theorem}{Theorem}[section]
\newtheorem{proposition}[theorem]{Proposition}
\newtheorem{lemma}[theorem]{Lemma}

\newtheorem{notation}[theorem]{Notation}

\newtheorem{definition}[theorem]{Definition}
\newtheorem{remark}[theorem]{Remark}

\newcounter{cpteur}

\newcommand{\PSL}{\wt{\rm PSL}}

\newcommand{\R}{\mathbb R}
\newcommand{\E}{\mathbb E}
\newcommand{\D}{\mathbb D}

\newcommand{\M}{\mathbb M}
\newcommand{\N}{\mathbb N}

\newcommand{\wt}{\widetilde}
\newcommand{\wh}{\widehat}

\newcommand{\g}{ g_{\rm euc}}

\def\rmd{\mathop{\rm d\kern -1pt}\nolimits}
\def\rme{\mathop{\rm e\kern -1pt}\nolimits}

\def\bel{ \medskip
 \centerline{$ \ast \hbox to 1.0cm{}\ast \hbox to 1.0cm{}\ast $}
}

\newcommand{\sd}{\mathbb{S}^2}
\newcommand{\hi}[1]{\mathbb{H}^#1}

\newcommand{\ov}[1]{\overline{#1}}

\let\leq=\leqslant
\let\geq=\geqslant

\begin{document}

\title[reflection principle]{classical Schwarz reflection principle for Jenkins-Serrin 
type minimal surfaces}

\author[R. Sa
Earp $\ $ and $\ $ E. Toubiana  ]{
 Ricardo Sa Earp and
Eric Toubiana}

 \address{Departamento de Matem\'atica \newline
  Pontif\'\i cia Universidade Cat\'olica do Rio de Janeiro\newline
Rio de Janeiro \newline
22451-900 RJ \newline
 Brazil }
\email{rsaearp@gmail.com}

\address{Institut de Math\'ematiques de Jussieu - Paris Rive Gauche \newline
Universit\'e Paris Diderot - Paris 7 \newline
Equipe G\'eom\'etrie et Dynamique,  UMR 7586 \newline
B\^atiment Sophie Germain \newline
Case 7012 \newline
75205 Paris Cedex 13 \newline
France}
\email{eric.toubiana@imj-prg.fr}

\thanks{
The  first and second authors were partially supported by CNPq
 of Brasil.
}

\date{\today}

\begin{abstract}
 We give a proof of the classical Schwarz reflection principle for Jenkins-Serrin type  
minimal surfaces in 
 the homogeneous  three manifolds $\E(\kappa,\tau)$ for $\kappa \leq 0$ and $\tau 
\geq 0$. 
In our previous paper we proved a reflection principle in Riemannian manifolds. 
The statements and techniques in the two papers are distinct. 
\end{abstract}


\maketitle

\section{Introduction}

In this paper we focus the classical Schwarz  
reflection principle across  a geodesic 
line  in the boundary of a minimal surface in $\R^3$  and more generally in 
three dimensional homogeneous spaces $\E(\kappa,\tau)$ for $\kappa<0$ and $\tau \geq 0$.

The Schwarz reflection principle was shown in some special cases. 
One kind of examples arise for  
the solutions of the classical Plateau problem in $\R^3$ containing a segment of a 
straight line in the boundary, see Lawson 
\cite[Chapter II, Section 4, Proposition 10]{Lawson}. Another kind occur for 
vertical graphs in $\R^3$ and $\hi2\times \R$ containing an arc of a horizontal geodesic, 
see \cite[Lemma 3.6]{SE-T4}.

On the other hand, there is no proof of the reflection principle for general minimal 
surfaces in $\R^3$ containing a straight line in its boundary.

\smallskip 

 The goal of this paper is to provide a proof of the reflection principle about vertical 
geodesic lines for Jenkins-Serrin type minimal surfaces in 
$\R^3$ and other three dimensional homogeneous manifolds such as, for example, 
$\hi2 \times \R$, $\PSL_2(\R, \tau)$ and $\sd \times \R$, see Theorem \ref{T.Reflexion}. 
The proof also holds for horizontal geodesic lines.

\smallskip 

We observe that this classical  Schwarz reflection principle was used by many authors, 
including the present authors, 
in 
$\R^3$ and $\hi2 \times \R$.
 
We recall that the authors proved another reflection principle for minimal surfaces in 
general three dimensional Riemannian manifold with quite different statement and 
techniques, see \cite{SE-T9}.

\smallskip 

We are 
grateful to the referee of our paper whose remarks greatly improved this
work.

\section{A brief description of the three dimensional homogeneous 
manifolds $\E(\kappa,\tau)$}
For any  $r>0$ we denote by $\D (r)\subset \R^2$ the open disc of  
$\R^2$ with center at the origin and with radius $r$ (for the Euclidean 
metric). 

For any  $\kappa \leq 0$ and  $\tau \geq 0$ we consider the model of  
$\mathbb{E}(\kappa,\tau)$ given by $\mathbb{D}(\dfrac{2}{\sqrt{-\kappa}}) \times \R$ 
equipped with the metric  
\begin{equation}\label{Eq.PSL}
 \nu_\kappa^2(dx^2 + dy^2) + 
 \left(\tau\nu_\kappa(ydx - xdy) + dt 
 \right)^2. 
\end{equation}
where $\nu_\kappa = \dfrac{1}{1 +\kappa\frac{x^2+y^2}{4}}$.
We observe that $ \mathbb{E}(-1,\tau)= \PSL_2(\R,\tau)$. 
By abuse of notations we set $\D (\frac{1}{0})= \D (+\infty)= \R^2$. Thus 
$ \mathbb{E}(0,\tau) = {\rm Nil}_3(\tau)$. Also, $\R^3$ equipped with the Euclidean 
metric is a model of $\mathbb{E}(0,0)$.


\smallskip 

We denote by $\M (\kappa)$ the complete, connected and simply connected Riemannian 
surface with constant curvature $\kappa$. Notice that for $\kappa<0$ a model of $\M 
(\kappa)$ is given by the disc $\mathbb{D}(\dfrac{2}{\sqrt{-\kappa}}) $ equipped 
with the 
metric 
$ \nu_\kappa^2(dx^2 + dy^2)$.

We recall that 
$\mathbb{E}(\kappa,\tau)$ is a fibration over $\M (\kappa)$, and the 
projection 
$\Pi : \mathbb{E}(\kappa,\tau) \longrightarrow \M (\kappa)$ is a Riemannian submersion, 
see for example \cite{Daniel}. Moreover the unit vertical field 
$\frac{\partial}{\partial t}$ is a Killing field generating a one-parameter group of 
isometries given by the vertical translations.

\smallskip

We have seen in \cite[Example 2.2-(2)]{SE-T9} that the horizontal geodesics and the 
vertical geodesics of $\E(\kappa,\tau)$ admit a reflection. That is, for any such a 
geodesic $L$, there exists a non trivial isometry $I_L$ of  $\E(\kappa,\tau)$ 
satisfying 
\begin{itemize}
\item $I_L$ is orientation preserving,

 \item $I_L (p) =p$ for any $p\in L$,
 
\item $ I_L \circ I_L= {\rm Id}$.
\end{itemize}

Let $\Omega$ be any domain of $\M(\kappa)$ and let $u:\Omega\longrightarrow\R$ be a 
$C^2$-function. We say that the set
$\Sigma:= \{(p, u(p)),\ p\in \Omega\} \subset \mathbb{D}(\dfrac{2}{\sqrt{-\kappa}}) \times 
\R$ is a {\em vertical graph}. Note that the Killing field $\frac{\partial}{\partial t}$ 
is transverse to $\Sigma$. Thus, by the well-known criterium of stability, if 
$\Sigma$ is a minimal surface then $\Sigma$ is stable. 

\smallskip

Consider some arbitrary  local coordinates 
$(x_1,x_2,x_3)$ of $\E(\kappa,\tau)$. 
Let $u$ be a $C^2$ function defined on a domain $\Omega$  
contained in the 
$x_1,x_2$ plane of coordinates. Let $S\subset \E(\kappa,\tau)$ be the graph of $u$. 
Then $S$ is a minimal surface if $u$ satisfies an elliptic PDE (called 
{\em minimal surface equation})
\begin{equation*}
  F(x,u,u_1,u_2,u_{11}, u_{12}, u_{22})=0,
\end{equation*}
see \cite[Equation (13)]{SE-T9}. Furthermore, if $u$ has bounded gradient then the 
PDE is uniformly elliptic.

\section{Jenkins-Serrin type minimal surfaces}\label{S.J-S}

The original Jenkins-Serrin's theorem was conceived in $\R^3$, see 
\cite[Theorems 1, 2 and 3]{J-S}. It was extended in $\hi2 \times\R$ by 
B. Nelli and H. Rosenberg \cite[Theorem 3]{N-R} and in $\M^2\times \R$ by 
A.L. Pinheiro \cite[Theorem 1.1]{Pinheiro} where $\M^2$ is a complete Riemannian surface. 
Later on  it was established in $\PSL_2(\R)$ by R. Younes 
\cite[Theorem 1.1]{Younes} and in $\text{Sol}_3$ by M. H. 
Nguyen 
\cite[Section 3.6]{Nguyen}.
As a matter of fact the same  proof also works in the 
homogeneous spaces 
$\E(\kappa,\tau)$ for any $\kappa < 0$ and $\tau \geq 0$.

\smallskip 

We state briefly below the Jenkin-Serrin type theorem in the homogeneous 
spaces $\E(\kappa,\tau)$ for $\kappa<0$ and $\tau \geq 0$ (same statement holds in 
 $\R^3$ and in $\M^2\times\R$).

\smallskip

 Let $\Omega$ be a bounded convex  domain in $\M^2(\kappa)$, thus for any point 
 $p\in \Gamma := \partial \Omega$ there 
is 
a complete geodesic line $\Gamma_p\subset \M^2(\kappa)$ such that $\Omega$ remains in one open 
component of 
$\M^2(\kappa) \setminus \Gamma_p$.

We assume that the  $C^0$ Jordan curve $ \Gamma \subset \M^2(\kappa)$ is 
constituted of two families 
 of open geodesic arcs $A_1,\dots,A_a$, $B_1,\dots,B_b$ 
and a family of  open arcs 
$C_1,\dots,C_c$ with their endpoints. We assume also that no two $A_i$ and no two $B_j$ have a 
common endpoint. 

%

\smallskip

On each open arc $C_k$ we assign  a continuous boundary data $g_k$. 

Let $P\subset \ov \Omega$ be any polygon  whose vertices are chosen 
among the endpoints  of the open geodesic arcs 
$A_i, B_j$, we call $P$ an {\em admissible polygon}. 
We set 
\begin{equation*}
 \alpha (P) = \sum_{A_i\subset P} \lVert A_i\lVert,\ 
 \beta (P)  = \sum_{B_j\subset P} \lVert B_j\lVert,  \ 
 \gamma(P)= \text{perimeter of }\ P.
\end{equation*}
With the above notations the Jenkins-Serrin's theorem asserts the following:

If the family $\{C_k \}$ is not empty then there exists a function \linebreak
$u: \Omega \longrightarrow ~\R$ whose graph is a minimal surface in 
$\E(\kappa,\tau)$ 
and such that 
\begin{equation*}
 u_{\vert A_i}=+\infty,\  u_{\vert B_j}=-\infty,\ 
 u_{\vert C_k}=g_k
\end{equation*}
if and only if 
\begin{equation} \label{H.J-S}
 2\alpha (P) < \gamma (P),\ \ 
  2\beta (P) < \gamma (P)
\end{equation}
for any admissible polygon $P$. In this case the function $u$ is unique.  

If the family $\{C_k \}$ is empty  such a function $u$ exists  if and only if 
$\alpha(\Gamma) = \beta (\Gamma)$ and condition (\ref{H.J-S}) holds for any  
admissible polygon $P\not= \Gamma$. In this case the function $u$ is unique up to an 
additive constant.

\smallskip 

We denote by $\Sigma \subset \E(\kappa,\tau)$ the graph of $u$ over $\Omega$ and we call 
such a surface 
a {\em Jenkins-Serrin type minimal surface}. 

\begin{remark}\label{R.boundary}
We observe that when the family $\{C_k \}$ is 
empty, the boundary of $\Sigma$ is the union of vertical geodesic line $\{q\} \times \R$ 
for any common endpoint $q$ between  geodesic arcs $A_i$ and 
$B_j$.

Suppose that the family $\{C_k \}$ is not empty and let $x_0$ 
be a common vertex  between 
$A_i$ and $C_k$, if any. If $g_k$ has a finite limit at $x_0$, 
say $\alpha$, then the 
half vertical line $\{x_0  \}\times [\alpha, +\infty[\,$  lies 
in the boundary of 
$\Sigma$. Now if $x_0$  is a common vertex  between 
$B_j$ and $C_k$ and if $g_k$ has a finite limit at $ x_0$, say 
$\beta$, then 
the 
half vertical line $\{x_0  \}\times \,]-\infty, \beta]$  lies 
in the boundary of 
$\Sigma$. At last, if $x_0$ is a common vertex  between 
$C_i$ and $C_k$ and  if $g_i$ and $g_k$ have different finite limits at 
$x_0$, say 
$\alpha < \beta$, then the vertical segment 
$\{x_0  \}\times [\alpha, \beta]$ lies 
in the boundary of $\Sigma$.
\end{remark}

\section{Main theorem}

For any vertical geodesic line $L$ of $\E(\kappa,\tau)$, we denote by $I_L$  
the reflection about the line 
$L$.

\bigskip

\begin{theorem}\label{T.Reflexion}
Using the notations of section \ref{S.J-S} and under the 
assumptions of Remark 
\ref{R.boundary}, let 
 $\gamma \subset \{x_0 \}\times \R :=   L \subset \E(\kappa,\tau)$
 be a 
vertical component of the boundary of the  open minimal vertical graph 
$\Sigma \subset \E(\kappa,\tau)$, where $\kappa<0$ and $\tau \geq 0$.

Then, we can extend minimally  $ \Sigma$ by reflection about $L$. More precisely,  
$S :=  \Sigma\, \cup \gamma \cup I_L (\Sigma)$ is a smooth minimal surface invariant by 
the reflection about $\Gamma$, containing 
${\rm int}( \gamma)$ in its interior. 

Furthermore the same statement and proof hold for $\Sigma\subset \R^3$ or 
$\Sigma\subset\sd \times \R$.
\end{theorem}
Observe that the possible cases for $\gamma$ are the following: the whole line 
$L$,  a half line of $L$ or a closed  geodesic arc of $L$.

 Observe also that, since we are under the assumptions of Remark \ref{R.boundary}, 
if $x_0$ is an endpoint of some arc $C_i$, then $g_i$ has a  finite limit at $x_0$.

\begin{remark}\label{R.horizontal}
We use the same notations as in Theorem \ref{T.Reflexion}.
Suppose that the boundary of $\Sigma$ contains an open arc 
$\delta$ (graph over an arc $C_k$) of a horizontal geodesic line 
$\Upsilon$ of $\PSL_2(\R,\tau)$.

We denote by $I_\Upsilon$ the reflection in $\PSL_2(\R,\tau)$ about $\Upsilon$.

We can prove as in \cite[Lemma 3.6]{SE-T4} $($in $\hi2 \times \R)$ that we can 
extend $ \Sigma$ by 
reflection about $\Upsilon$:  $ \Sigma\, \cup\, \delta \cup I_\Upsilon (\Sigma)$ 
is a connected smooth minimal surface containing $ \delta$ in its interior.
The same observation holds also in Heisenberg space and $\sd \times \R$.

On the other hand, we can verify that the proof of Theorem \ref{T.Reflexion} 
also works for reflection about horizontal geodesic lines.
\end{remark}

\begin{proof}
 For the sake of clarity and simplicity of notations, we provide the proof 
 in $\PSL_2(\R,\tau)=\E(-1,\tau)$. Nevertheless, all arguments and 
constructions hold in 
$\E(\kappa,\tau)$ for any $\kappa <0$ and $\tau \geq 0$,   in 
$\R^3$, that is for $\kappa=\tau=0$ and in $\sd \times \R$, that is 
for $\kappa=1$ and $\tau=0$.

\smallskip 

 We assume that the family $C_k$ is not empty. The other situation can be handled in a 
 similar way.

\smallskip 

 Recall that, by assumption, if $x_0$ is an endpoint of some arc  
$C_i$ (if any), then $g_i$  has a  finite limit at $x_0$.

\smallskip 

We suppose  that all functions $g_k$ admit also a 
 finite limit at the endpoints of $C_k$ different of $x_0$  (if any). It is possible  
to 
carry out a proof without this assumption but 
the details are cumbersome,  as we can see in the following.

 Suppose for instance that $x_1$  ($\not= x_0$) is an endpoint of some arc $C_k$, 
that $g_k$ has 
no limit at $x_1$ and that $g_k$ is bounded near $x_1$. 
Setting $\alpha=\frac{1}{2}(\liminf_{x \to x_1} g_k(x) + \limsup_{x \to x_1} g_k(x))$, 
we can find a sequence $(p_n)$ on   $C_k$ such that 
\begin{equation*}
 p_n \to x_1 \quad \text{and} \quad g_k(p_n)=\alpha \text{ for any } n.
\end{equation*}
Then we consider the new function $g_{k,n}$ on $C_k$ setting $g_{k,n}(x)=\alpha$ on the segment 
$[x_1,p_n]$ of $C_k$ and $g_{k,n}=g_k$ outside this segment. Now the continuous  function $g_{k,n}$ 
has a limit at $x_1$. Observe that for any $x\in C_k$ we have $g_{k,n}(x)=g_k(x)$ for any $n$ large 
enough.

 If $g_k$ is not bounded near $x_1$, we first truncate, for any 
$n>0$, the  
function $g_k$ above by $n$ and below by -$n$. We obtain a new continuous and bounded function 
$h_{k,n}$ on $C_k$. Then we proceed as above.

\smallskip

 For any integer $n$ we consider the Jordan curve $\Gamma_n$ obtained by the 
union of the geodesic arcs $A_i$ at height $n$,  
the geodesic arcs  
$B_j$ at height 
$-n$,  the graphs of functions $g_k$  over the open arcs $C_k$ (or $g_{k,n}$ if $g_k$ 
has no finite limit at some endpoint of $C_k$), 
 and  the 
vertical segments 
necessary to form a Jordan curve. Thus $\Gamma$ is the  projection of 
$\Gamma_n$ on $\hi2$.

Let  $\Sigma_n\subset \PSL_2(\R,\tau)$ be the embedded area minimizing disc with boundary 
$\Gamma_n$ given by Proposition \ref{P.embedded} in the Appendix. We have 
\begin{itemize}
 \item $\Sigma_n\subset \ov \Omega \times \R$,
 
 \item $\mathring{\Sigma}_n:=\Sigma_n\setminus \Gamma_n=\Sigma_n \cap 
(\Omega\times \R)$ is a vertical graph over $\Omega$.
\end{itemize}

\smallskip 

We set $\gamma_n := \Sigma_n \cap L$, where  $L:=\{x_0\} \times \R$, thus 
$\gamma_n \subset \gamma$ for any $n$. 
Due to the fact that $\Sigma_n$ is area minimizing we can apply the reflection 
principle 
about the vertical line $L$, 
 this is proven in detail in  \cite[Proposition 3.4]{SE-T9}. Thus, 
$S_n:=  \Sigma_n \cup I_L ( \Sigma_n)$ is an embedded minimal surface containing 
${\rm int}(\gamma_n)$ in its interior. By construction $S_n$ is invariant under the 
reflection $I_L$ and is orientable.

\smallskip

Let  $u_n : \Omega \longrightarrow \R$ be the function whose the graph is 
$\mathring{\Sigma}_n$. Thus $u_n$ extends continuously by  $n$ 
on the edges ${\rm int} (A_i)$, by $-n$ on the edges ${\rm int} (B_j)$ and by 
$g_k$  (or $g_{k,n}$) over the open arcs $C_k$. Using the lemmas derived in 
\cite{Younes}, following the 
original proof of \cite[Theorem 2]{J-S}, it can be proved that, 
 up to considering a subsequence, the sequence of functions 
$(u_n)$ converges  to a function $u:\Omega\longrightarrow \R$ in the $C^2$-topology, 
uniformly over any compact subset of $\Omega$.

\smallskip

Let $d_n$ be the intrinsic distance on $S_n$. For any $p\in S_n$ and any $r>0$ we denote 
by $B_n(p,r)\subset S_n$ the open
geodesic disc of $S_n$ centered at $p$ with radius $r$. By construction, for 
any $p\in {\rm int}(\gamma)$ there exist $n_p\in \N$ and a real 
number $c_p>0$ such that for any 
integer  
$n\geq n_p$ we have $p\in {\rm int}(\gamma_n)\subset {\rm int} (S_n)$ and 
$d_n (p, \partial  S_n) > 2 c_p$.

\smallskip 

We assert that  the Gaussian curvature $K_n$ of the surfaces $S_n$ is uniformly bounded 
in the neighborhood of each point of ${\rm int}(\gamma)$, 
independently of $n$.  

\begin{proposition}\label{P.courbure-fibre}
For any $p\in {\rm int}(\gamma)$ there exist  $R_p, K_p >0$, and 
there exists 
 $n_p \in \N $ satisfying $p \in  {\rm int} (\gamma_{n_p}) \subset S_{n_p}$ and 
$d_{n_p}(p,\partial S_{n_p}) > 2R_p$, 
such that for any integer $n\geq n_p$ 
we have $p\in {\rm int} (\gamma_n )\subset S_n$ and
\begin{equation*}
 \lvert K_n (x) \lvert \leq K_p,
\end{equation*}
 for any 
$x\in B_n (p,R_p)$.
\end{proposition}

 We postpone the proof of Proposition \ref{P.courbure-fibre} until 
Section \ref{S.Prop}.


\vskip1mm

Assuming Proposition \ref{P.courbure-fibre} we will prove that for any 
$p \in {\rm int}(\gamma)$  there exists an embedded minimal disc $D(p)$, 
containing $p$ in its interior,  such that 
$ D(p) \subset\Sigma \cup \gamma \cup I_L(\Sigma)$. 
This will prove that  $\Sigma \cup \gamma \cup I_L(\Sigma)$ is a  minimal surface, that 
is smooth along ${\rm int} (\gamma)$.

\smallskip 

Let  $p\in  {\rm int}(\gamma)$, we deduce from Proposition 
\ref{P.courbure-fibre} that 
there exist real numbers $R_p, K_p >0$ and $n_p\in \N$ such that for any integer $n \geq 
n_p$ and for any point $x \in B_n(p,R_p)$ we have $\lvert K_n(x)\lvert \leq K_p$. 
 By construction, $B_n(p,R_p)$ is an embedded minimal disc.

\vskip1mm

 Therefore it can be proved as in \cite{RST}, using 
\cite[Proposition 2.3, Lemma 2.4]{RST} and the discussion that follows, 
 that up to taking a subsequence,
the geodesic discs $B_n(p,R_p)$ 
converge 
for the $C^2$-topology to 
a minimal disc $D(p)\subset \R^3$ containing $p$ in its interior. We recall that 
each 
 geodesic disc $B_n(p,R_p)$ is embedded, contains an open subarc $\gamma(p)$ of 
$\gamma$ (which does 
not depend on $n$) passing through $p$, and 
$B_n(p,R_p)$ is invariant under the reflection $I_L$. Thereby the minimal disc 
$D(p)$  
also  is embedded,  contains the subarc $\gamma(p)$ and inherits the same symmetry. 

\smallskip

We set $S:=\Sigma\cup \gamma \cup I_L(\Sigma)$.

By construction the surfaces ${\rm int}( S_n)\setminus L$ converge to 
$\Sigma \cup I_L (\Sigma)$. We observe that 
$B_n(p,R_p)\setminus L \subset S_n\setminus L$ and then 
$D(p) \setminus L \subset \Sigma \cup I_L(\Sigma)$. Then we have 
$D(p)\subset  S$. We conclude henceforth that $S$ is a smooth   
minimal surface invariant under the reflection $I_L$, this accomplishes the proof of the 
theorem.
\end{proof}

\medskip

\begin{remark}
 Theorem \ref{T.Reflexion} holds also in case where $x_0$ is and endpoint of some arc 
 $C_i$ and $g_i$ has an {\em infinite limit} at $x_0$.
 
  Indeed, assume that $\lim_{x \to x_0} g_i(x)= +\infty$. We denote by $g_{i,n}$ the new function on 
$C_i$ obtained by  truncating the function $g_i$ above by $n$. Then, in the proof of  Theorem 
\ref{T.Reflexion}, we consider the embedded area minimizing disc $\Sigma_n$ constructed with the 
function $g_{i,n}$ on $C_i$ $($instead of $g_i)$. Then we can proceed the proof in the same way. 
\end{remark}

\begin{remark}
We don't know if the Jenkins-Serrin type theorem was established in the 
Heisenberg spaces $Nil_3(\tau)=\E(0,\tau)$ for $\tau > 0$. Assuming the 
Jenkins-Serrin type theorem, the proof of Theorem \ref{T.Reflexion} works  
 to establish the same reflection principle  for vertical 
 geodesic lines  in $Nil_3(\tau)$.
\end{remark}

\section{Proof of Proposition \ref{P.courbure-fibre}}  \label{S.Prop}

We argue by absurd. 

Suppose by contradiction that there exists 
$p\in {\rm int}(\gamma)$ such that for any $k\in \N^*$ 
there exist an integer 
$n_k > k$ and  $x_k \in B_{n_k}(p, \frac{1}{k})$ such that 
 $ \lvert K_{n_k}(x_k) \lvert > k^2$.

There exist $c>0$ and $k_0 \in \N^*$ such that for any integer $k\geq k_0$ we have 
$p\in {\rm int} (\gamma_{n_k})$ and $d_{n_k}(p,\partial S_{n_k}) 
>2c$. Thus 
$\ov B_{n_k}(p,c) \subset {\rm int} (S_{n_k})$.

Moreover there exists an integer $k_1 > k_0$ such that for any integer 
$k \geq k_1$ we have  $d_{n_k}(x_k,\partial B_{n_k}(p,c)) > c/2$.

 \smallskip

From now on,  we are going to use  classical blow-up techniques.  

 \smallskip
 
 Define the continuous function 
 $f_k: \ov B_{n_k}(p,c) \longrightarrow [0, + \infty[\, $ for any  $k\geq k_1$, setting: 
  $f_k(x)=\sqrt{\lvert K_{n_k}(x)\lvert} \,d_{n_k}(x,\partial B_{n_k}(p,c))$. 
 
\smallskip 

Clearly $f_k\equiv 0$ on $\partial B_{n_k}(p,c)$ and 
\begin{equation*}
 f_k(x_k)=\sqrt{\lvert K_{n_k} (x_k)\lvert}\, d_{n_k}(x_k,\partial B_{n_k}(p,c))
 \geq k \frac{c}{2}.
\end{equation*}
We fix a point $p_k\in  B_{n_k}(p,c)$ where the function $f_k$ attains its maximum 
value, hence 
\begin{equation}\label{minoration-f-fibre}
  f_k(p_k)\geq k\frac{c}{2}.
\end{equation}
We deduce therefore 
\begin{equation}\label{minoration-fibre}
 \lambda_k:=\sqrt{\lvert K_{n_k}(p_k)\lvert} \geq \frac{k c}{2d_{n_k}(p_k,\partial B_{n_k}(p,c))} 
 \geq \frac{kc}{2c}=\frac{k}{2} .
\end{equation}
\begin{notation}
 We set $\rho_k = d_{n_k}(p_k,\partial B_{n_k}(p,c))$ and we denote by 
$D_k \subset B_{n_k}(p,c) \subset S_{n_k}$ the open geodesic disc with center 
$p_k$ and radius $\rho_k /2$. Notice that $D_k$ is embedded.
\end{notation}
 For further purpose we emphasize that $D_k$ is 
an orientable minimal surface 
of $\PSL_2(\R,\tau)$.

\smallskip

Let us consider the model of $\PSL_2(\R,\tau)=\mathbb{E}(-1,\tau)$ given by 
(\ref{Eq.PSL}) for $\kappa=-1$, that is the product set $\D(2)\times\R$ equipped with 
the metric 
\begin{equation}\label{metric}
 ds^2:= \mu^2(dx^2 + dy^2) + 
 \big(\tau \mu (ydx - xdy) +dt  \big)^2
\end{equation}
where $\mu=\mu(x,y)= \dfrac{1}{1-\frac{x^2+y^2}{4}}$.

For any integer  $k\geq k_1$ we set 
$\mu_k= \mu_k(u,v)=\dfrac{1}{1-\frac{u^2+v^2}{4\lambda_k^2}}$.  
We consider, as in the Nguyen's thesis \cite[Section 2.2.3]{Nguyen}, the product set 
$\D(2\lambda_k) \times \R$ 
 equipped with the metric
\begin{equation}\label{metric-k}
 ds_k^2 :=  \mu_k^2(du^2 + dv^2) + 
 \left( \frac{\tau}{\lambda_k}\, \mu_k (vdu - udv) + dw \right)^2.
\end{equation}
Thus $(\D(2\lambda_k) \times \R,\, ds_k^2)$ is a model of  
$\E (\frac{-1}{\lambda_k^2}, \frac{\tau}{\lambda_k})$. 

\smallskip 

\begin{remark}
Since  $\PSL_2(\R,\tau)$ is a homogeneous space, for any integer $k\geq k_1$, up to considering an isometry 
of $\PSL_2(\R,\tau)$ which sends  
$p_k$ to the origin  $0_3:=(0,0,0)$, we can assume that $p_k=0_3$. 
See for example \cite[Chapter 5]{Penafiel} or \cite[Proposition 1.1.7]{Nguyen}.
\end{remark}

\smallskip

Let us consider the homothety 
\begin{align*}
 H_k : \D (2) \times \R & \longrightarrow \D(2\lambda_k) \times \R \\ 
 (x,y,t) &\longmapsto (u,v,w)= \lambda_k\, (x,y,t).
\end{align*}
We have  $H_k^*(ds_k^2)= \lambda_k^2\, ds^2$, see (\ref{metric}) and (\ref{metric-k}). 
Then, it follows that  
$\wt D_k := H_k (D_k)$ is an embedded minimal surface of 
$(\D(2\lambda_k) \times \R,\, ds_k^2)$. 

By construction, $\wt D_k$
is a geodesic disc with center the origine $0_3$ of 
$\D(2\lambda_k) \times \R$ : 
$0_3 \in \wt D_k \subset \D(2\lambda_k) \times \R$. Moreover the radius of 
$\wt D_k$ is $\wt \rho_k =\lambda_k \cdot$(radius of $D_k$), that is 
$\wt \rho_k=\lambda_k \, \rho_k/2 $.

\smallskip 

Using the estimate (\ref{minoration-f-fibre}) we get 
\begin{equation}\label{Eq.radius}
 \wt \rho_k=\lambda_k \, \rho_k/2 =\sqrt{\lvert K_{n_k}(p_k)\lvert}\,\, 
 d_{n_k}(p_k,\partial B_{n_k}(p,c))/2 = \frac{f_k(p_k)}{2} \geq \frac{kc}{4},
\end{equation}
thus  $ \wt \rho_k \to  \infty$ if $k \to \infty$.

\smallskip

Let $\g = du^2 + dv^2 + dw^2$ be the Euclidean metric of $\R^3$. We observe that 
$(\D(2\lambda_k) \times \R,\, ds_k^2)$ converges to $(\R^2\times\R, \, \g)$ for the 
$C^2$-topology, uniformly on any compact subset of $\R^3$.

\smallskip

 We denote by $\wt K_{n_k}$ the Gaussian curvature of $\wt D_k$. For any 
 $x\in D_k\subset \D(2) \times \R $, setting $X=H_k(x) \in \wt D_k \subset 
\D(2\lambda_k) \times \R$, we get $\wt K_{n_k} (X) = \frac{K_{n_k}(x)}{\lambda_k^2}$. 
Hence for any $X\in \wt D_k$ we obtain 
\begin{align}
 \sqrt{\lvert\wt K_{n_k} (X)\lvert} =\frac{\sqrt{\lvert K_{n_k}(x)\lvert}}{\lambda_k}
 &\leq  \frac{f_k(p_k)} {\lambda_k  \,d_{n_k}(x,\partial B_{n_k}(p,c)) } \notag \\
 &=   \frac{d_{n_k}(p_k,\partial B_{n_k}(p,c))}{d_{n_k}(x,\partial B_{n_k}(p,c))} 
< 2, \label{estime-fibre}
\end{align}
since $d_{n_k}(x,\partial B_{n_k}(p,c)) >\dfrac{\rho_k}{2}$. 

Furthermore, for any integer $k\geq k_1$ we have 
\begin{equation}\label{origine-fibre}
 \sqrt{\lvert\wt K_{n_k} (0_3)\lvert}=\frac{\sqrt{\lvert K_{n_k} 
(p_k)\lvert}}{\lambda_k}=1.
\end{equation}

We summarize  some facts derived before: 
\begin{lemma}
 \begin{itemize}
 \item each  $\wt D_k$ is an embedded and orientable minimal surface of 
$(\D(2\lambda_k) \times \R,\, ds_k^2)= 
\E(-\frac{1}{\lambda_k^2},\frac{\tau}{\lambda_k})$,

\item there is a uniform estimate of Gaussian curvature, see $(\ref{estime-fibre})$,

\item the radius $\wt \rho_k$ of the geodesic disc $\wt D_k$ go to $+\infty$ if
$k \to \infty$, see $(\ref{Eq.radius})$, 

\item the metrics $ds_k^2$ converge to $\g$ for the $C^2$-topology, uniformly on any 
compact subset of $\R^3$, see $(\ref{metric-k})$.
\end{itemize}
\end{lemma}
 Therefore it can be proved as in \cite{RST} (using 
\cite[Proposition 2.3, Lemma 2.4]{RST} and the discussion that follows), 
 that up to considering a subsequence,
 the $\wt D_k$ converge for the 
$C^2$-topology to a complete, connected and 
orientable minimal surface $\wt S$ of $\R^3$.

\begin{remark}\label{R.disc} From the construction described in 
\cite{RST},  the surface $\wt S$ has the following properties.

There exist $r, r_0>0$ such that for any $q\in \wt S$, a piece $\wt G (q)$ of $\wt S$,  
containing the geodesic disc with center  $q$ and radius $r_0$, is a graph over 
the open 
disc $D(q,r)$ of $T_q \wt S$ with center $q$ and 
radius $r$ $($for the Euclidean metric of $\R^3)$. Furthermore:
\begin{itemize}
 \item for $k$ large enough, a piece 
$\wt G_k (q)$ of $\wt D_k\subset \D(2\lambda_k) \times \R$ is also a graph over $D(q,r)$ and the 
surfaces $\wt G_k(q)$ 
converge for the $C^2$-topology to $\wt G(q)$,

\item for any $y\in \wt G (q)$ there exists $k_y\in \N$ such that for any  
$k\geq k_y$ we can choose the piece $\wt G_k (y)$ of $\wt D_k$ such that 
$\wt G_k (q) \cup \wt G_k(y)$ is connected.
\end{itemize}
\end{remark}

\smallskip 

By construction we have $0_3\in \wt S$ and, denoting by $ \wt K$ the Gaussian curvature 
of $\wt S$ in $(\R^3, \g)$, we deduce from (\ref{origine-fibre})  
\begin{equation}\label{origine-S-fibre}
 \lvert \wt K(0_3)\lvert =1.
\end{equation}
For any integer  $k\geq k_1$ we set  $\wt L_k:= H_k (L)$. Thus, 
$\wt L_k$ is a vertical straight line of $\R^3$.
\begin{definition}
Let  $\delta_k$ be the distance in $\D(2\lambda_k) \times \R$ induced by the metric  
$ds_k^2$. 

We say that the sequence of vertical lines $(\wt L_k)$ in $\R^3$ 
{\em  disappears to infinity} if $\delta_k (0_3, \, \wt L_k) \to +\infty$ when 
$k \to +\infty$
\end{definition}
There are two possibilities: the sequence $(\wt L_k)$  disappears or not to infinity. 
We are going to show that either case cannot occur, we will find therefore a 
contradiction.

\smallskip


\noindent {\em First case:} $(\wt L_k)$  disappears to infinity.

Observe that, by construction, the geodesic discs $B_{n_k}(p,c)$ are invariant 
under the 
reflection $I_L$ and $f_k(q)=f_k( I_L(q))$ for any $q\in B_{n_k}(p,c)$.  Since  
$p_k=0_3$  by assumption, we can assume that 
$0_3\in \Sigma_{n_k} \subset S_{n_k}$ for any $k\geq k_1$.

Let $q \in \wt S$, and consider a minimizing geodesic arc $\delta\subset \wt S$ joining 
$0_3$ to $q$. It follows from Remark \ref{R.disc} that there exist a finite 
number of 
points $q_1=0_3,\dots, q_n=q$  belonging to $\delta$, and there exists $k_q \in \N$ 
such that:
\begin{itemize}
 \item for any integer $k \geq k_q$ the subset $\cup_j \wt G_k (q_j)\subset \wt D_k$ is 
 connected and converges for the $C^2$-topology to the subset 
 $\cup_j \wt G (q_j)\subset \wt S$,
 
\item  for any integer $k \geq k_q$ we have 
\big($\cup_j \wt G_k (q_j)\big)\cap \wt L_k=\emptyset$.
\end{itemize}
Thus for any integer $k \geq k_q$ we obtain that 
$H_k^{-1}(\cup_j \wt G_k (q_j)) \cap L=\emptyset$, that is 
$H_k ^{-1}(\cup_j \wt G_k (q_j))\subset D_k\cap \Sigma_{n_k}$.

Setting $\wh D_k := H_k (D_k \cap \Sigma_{n_k})$ we deduce that 
the sequence $(\wh D_k)$ converges to $\wt S$ for the $C^2$-topology too. 
 Observe that $\wh D_k \cap \wt L_k \subset \partial \wh D_k $. Therefore 
 any minimal surface  $\wh D_k \setminus \wt L_k $ is a Killing graph and 
thus  $\wh D_k $ is a stable minimal surface of 
$\E (\frac{-1}{\lambda_k^2}, \frac{\tau}{\lambda_k})$. 

Therefore it can be proved as in the discussion following Lemma 2.4 in 
\cite{RST} that $\wt S$ is a connected, complete, orientable and stable minimal 
surface of $\R^3$. 
Thanks to results of do Carmo and Peng 
\cite{[DoC-Pe]},
Fischer-Colbrie and Schoen \cite{[FC-Sc]} and Pogorelov \cite{[Pog]}, 
$\wt S$ is a plane. But this gives a contradiction with the 
 curvature relation 
(\ref{origine-S-fibre}).

\medskip

\noindent {\em Second case:} $(\wt L_k)$ does not  disappear to infinity.

We will prove that the Gauss map of $\wt S$ omits infinitely many points, hence 
$\wt S$ would be a plane  
(see \cite[Corollary 1.3]{Fujimoto} or \cite{Xavier}), 
contradicting the curvature relation (\ref{origine-S-fibre}).

\smallskip 

Let $\alpha \in (0,\pi]$ be the interior angle of $\Gamma$ at vertex $x_0$, 
 ($\alpha$ exists since $\Gamma$ is the boundary of a convex domain). 
Observe that 
the case where $\alpha=\pi$ is under consideration.

\smallskip 

Since $\Omega$ is convex, there exists a geodesic line 
 $C_{x_0} \subset \hi2$ at $x_0$ such that $ C_{x_0}\cap \Omega=\emptyset$.
 Let $\Pi$ be the product $ C_{x_0} \times \R$ in 
 $(\D (2)  \times \R, ds^2)=\E(-1,\tau)$. When $\tau=0$ notice 
that $\Pi$, is a vertical 
totally geodesic plane in $\hi2 \times \R$. We recall that there are no totally 
geodesic surfaces in $\E(-1,\tau)$ if $\tau \not=0$, see 
\cite[Theorem 1]{Souam-Toubiana}.

Under our assumption, up to 
considering a subsequence, we can assume that the sequence $(\wt L_k)$ converges to a 
vertical 
straight line $\wt L \subset \R^3$ and that  $ \big(H_k (\Pi)\big)$ converges 
to a vertical plane $\wt \Pi \subset \R^3$ 
containing $\wt L$. Let us denote by $\wt \Pi^+$ and $\wt \Pi^-$ 
the two open halfspaces of $\R^3$ 
bounded by $\wt \Pi$.

\smallskip 

 \setcounter{cpteur}{1}

\begin{lemma}\label{L.plane} 
We have $(\wt S \cap \wt \Pi)\setminus \wt L=\emptyset$.
\end{lemma}

\begin{proof}
 Otherwise assume there exists a point 
$q\in \wt S \cap \wt \Pi$ such that $q\not\in \wt L$.  From the structure of the 
intersection of two minimal surfaces tangent at a point, see \cite[Theorem 7.3]{C-M} or 
\cite[Lemma, p. 380]{Serrin}, we may suppose that $\wt \Pi$ is 
transverse to $\wt S$ at $q$. Thus there is an open piece $\wt F(q)$ of $\wt S$ 
containing $q$ which is transverse to the plane $\wt \Pi$. Hence, for any integer $k$ 
large enough, a piece $\wt F_k(q)$   of $\wt D_k \subset \D(2\lambda_k) \times \R$ 
is so close to $\wt F(q)$ that it is 
transverse to $\wt \Pi$ too. Consequently we would have 
${\rm int} (\wt D_k )\cap \big(H_k (\Pi)\setminus \wt L_k\big) \not= 
\emptyset$, that is ${\rm int} (D_k)  \cap (\Pi \setminus L)\not= \emptyset$. But by 
construction we have ${\rm int}(S_{n_k}) \cap (\Pi \setminus L)=\emptyset$, which leads to 
a contradiction since $D_k \subset S_{n_k}$.
\end{proof}
\medskip 

\begin{lemma}\label{L.Intersection} 
 We have $\wt S \cap \wt \Pi =\wt L$.
\end{lemma}

\begin{proof}
Assume first that $\wt S \cap \wt \Pi=\emptyset$. Hence $\wt S$ stay in an open  
halfspace, say 
$\wt \Pi^+$, of $\R^3$ bounded by $\wt \Pi$. Observe that the halfspace $\wt 
\Pi^+$ is 
the limit of  open subspaces $H_k(\Pi^+)$ of 
$\D(2\lambda_k)\times \R$ where 
$\Pi^+$ is one of the two open halfspaces of $\D (2) \times \R$ bounded by 
$\Pi$. 
Consequently 
$\wt S$ is the limit of the graphs $\wt D_k \cap H_k (\Pi^+)$. Therefore, as 
in the first case, we obtain that $\wt S$  is stable and thus is a plane, giving a 
contradiction with (\ref{origine-S-fibre}). We obtain therefore 
$\wt S \cap \wt \Pi\not=\emptyset$.

\smallskip 

Let $q\in \wt S \cap \wt \Pi$. We deduce from Lemma \ref{L.plane} that $q\in \wt L$.
If $\wt \Pi$ were the tangent plane of $\wt S$ at $q$, 
then the intersection $\wt S \cap \wt \Pi$ would consist in a even number $\geq 4$ of 
arcs issued from $q$,  see \cite[Theorem 7.3]{C-M} or 
\cite[Lemma, p. 380]{Serrin}. Then we infer that 
$\wt S \cap (\wt \Pi \setminus \wt L)\not=\emptyset$ which is not possible due to 
 Lemma \ref{L.plane}. 
 
 Thus $\wt \Pi$ is transverse to $\wt S$ at $q$. Since $\wt S\cap \wt \Pi \subset \wt L$, 
 we deduce from  Lemma \ref{L.plane} that $\wt S\cap \wt \Pi$ contains an open arc of 
$\wt L$ containing 
$q$. This proves that  $\wt S\cap \wt \Pi$ contains a segment of 
$\wt L$. It is well known that if a complete minimal surface of $\R^3$ contains a segment 
of a straight line then it contains the whole straight line, see 
Proposition \ref{P.Line} in the Appendix. 
We conclude that $\wt S\cap \wt \Pi=\wt L$ as desired.
\end{proof}

\begin{remark}
To prove that $\wt S \cap \wt \Pi\not=\emptyset$
 we can alternatively argue as follows. Assume that $\wt S \cap \wt \Pi=\emptyset$. 
By construction $\wt S$ is 
a complete and connected minimal surface  in $\R^3$ without self-intersection. 
Furthermore we deduce from 
the estimates 
$(\ref{estime-fibre})$ that $\wt S$ has bounded curvature. It follows from 
\cite[Remark]{Harold} that $\wt S$ is properly embedded. Since $\wt S$ lies in a 
halfspace, we deduce from the halfspace theorem \cite[Theorem 1]{HM} that  
$\wt S$ is a 
plane, which gives a 
contradiction with  $(\ref{origine-S-fibre})$. Thus 
$\wt S \cap \wt \Pi\not=\emptyset$.
\end{remark}

\smallskip

We deduce from Lemma \ref{L.Intersection} that $\wt S\setminus \wt \Pi=\wt 
S\setminus \wt L$ has 
two connected components, say $\wt S^- \subset \wt \Pi^-$ and 
$\wt S^+ \subset \wt \Pi^+$. In the same way we denote by $ \Pi^+$ and $\Pi^-$ 
the 
two open halfspaces of $\D (2) \times \R$ bounded by $\Pi$. We have that 
$\wt \Pi^+$  (resp. $\wt \Pi^-$ ) is the limit of $H_k (\Pi^+)$ 
(resp. $H_k (\Pi^-)$).

We  set $D_k^{\pm} := D_k \cap \Pi^{\pm}$ and 
 $\wt D_k^{\pm}:= H_k (D_k^{\pm})=\wt D_k\cap H_k(\Pi^\pm) $. We observe 
that $ \wt D_k^+$ 
 and $ \wt D_k^-$ are vertical graphs and that $\wt S^+$ (resp. $\wt S^-$)
is the limit of $(\wt D_k^+)$ (resp. $(\wt D_k^-$))  for the $C^2$-topology.

For any integer $k\geq k_1$ we denote by $\wt N^k$ a smooth unit normal vector field on 
$\wt D_k$ with respect to the metric $ds_k^2$, see (\ref{metric-k}). Let $\wt N_3^k$ be 
the {\em vertical component} of $\wt N^k$, this means that 
$\wt N^k -\wt N_3^k\frac{\partial}{\partial t} $ and $\frac{\partial}{\partial t}$ are 
orthogonal  vector fields along $\wt D_k$. 

Since  $(\wt D_k)$ converges to 
$\wt S$ for the $C^2$-topology, we can define a unit 
normal field $\wt N$ on $\wt S$ as the limit of the fields $\wt N^k$.

\smallskip 

\begin{lemma}\label{L.Gauss}
 We have $\wt N_3\not=0$ on $\wt S^+ \cup \wt S^-$. Furthermore $\wt S^+$ and 
$\wt S^-$ are vertical graphs.
\end{lemma}

\begin{proof}
 Indeed,  we know that $ \wt D_k^+$ is a vertical graph. So we can 
assume  that $\wt N_3^k >0$ along $ \wt D_k^+$ for any $k \geq k_1$. By considering the 
limit of the fields $\wt N^k$ we get that $\wt N_3 \geq 0$ on $\wt S^+$. 

Let $q\in \wt S^+$ be a point such that $\wt N_3 (q)=0$, if any. Recall that the Gauss map 
of 
a non planar 
minimal surface of $\R^3$ is an open map. Therefore, in any neighborhood of $q$ in 
$\wt S^+$ it would exist points $y \in \wt S^+$ such that $\wt N_3(y) <0$, which leads to 
a 
contradiction. 

Thus we have $\wt N_3\not=0$ on $\wt S^+$. We prove in the same way that 
$\wt N_3\not=0$ on $\wt S^-$ too. 

\smallskip 

Assume by contradiction that $\wt S^+$ is not a vertical graph. Then there exist two 
points $q, \ov q \in \wt S^+$ lying to same vertical straight line. As the tangent planes 
of $\wt S^+$ at $q$ and $\ov q$ are not vertical, there exists a real number $\delta>0$ 
such that a neighborhood $V_{\ov q} \subset \wt S^+$ of  $\ov q$ and a 
 neighborhood $V_q \subset \wt S^+$ of  $ q$ are vertical graphs over an 
Euclidean disc 
of radius $\delta$ in the $(u,v)$-plane.

But, by construction, for $k$ large enough a piece $U_{\ov q}$ of $\wt D_k^+$ is 
$C^2$-close of $V_{\ov q}$ and a piece $U_q$ of $\wt D_k^+$ is 
$C^2$-close of $V_q$. Clearly this would imply that the vertical projections of 
$U_{\ov q}$ and $U_q$ on the $(u,v)$-plane have non empty intersection. But this is not 
possible since  $\wt D_k^+$ is a vertical graph. We conclude therefore that 
$\wt S^+$ is a vertical graph.

We can prove in the same way that $\wt S^-$ is a vertical graph.
\end{proof}
\vskip1mm

\noindent {\bf End of the proof of the proposition}

 \smallskip
Let $P\subset \R^3$ be any vertical plane verifying $\wt L\subset P$ and 
$P\not= \wt \Pi$. We deduce from Lemmas \ref{L.Intersection} and \ref{L.Gauss} that 
$(\wt S\cap P)\setminus \wt L$ is a vertical graph.
 Therefore, the structure of the intersection of two minimal surfaces 
tangent at a 
point, see \cite[Theorem 7.3]{C-M} or \cite[Lemma, p. 380]{Serrin}, shows that there 
cannot be two distinct points of $\wt L$ where the 
tangent plane of $\wt S$ is $P$. 

Let $\nu$ and $-\nu$ be the two unit vectors orthogonal 
to $P$. Since $\wt N_3\not=0$ on $\wt S \setminus \wt L$ we deduce that $\nu$ and $-\nu$ 
are 
not both assumed by the Gauss map of $\wt S $. By varying the 
vertical planes $P$, we obtain that the Gauss map of 
$\wt S $ omits infinitely many points (belonging to the equator of the 
2-sphere). Then 
$\wt S $ must be a plane, see \cite[Theorem]{Xavier} or \cite[Theorem I]{Fujimoto}. On 
account of (\ref{origine-S-fibre}) we arrive to a   contradiction. This accomplishes the 
proof of the proposition. \qed

\smallskip 

Observe that, actually, the proof of Proposition \ref{P.courbure-fibre} 
demonstrates the following result.
 
\begin{proposition}
Let $\Omega \subset  \M^2(\kappa)$, $\kappa \leq 0$, be a convex domain. 
 Let $C_1, C_2 \subset \partial \Omega$ 
be two open arcs admitting a common  endpoint $x_0\in \partial \Omega$.

Let $(\Sigma_n)\subset \E(\kappa,\tau) $, $\tau \geq 0$, be a sequence of minimal surfaces 
satisfying for any $n$: 
\begin{itemize}
 \item $\Sigma_n$ is a minimal surface with boundary, whose  interior is the graph of a 
function $u_n$ defined on  $\Omega $,

\item the function $u_n$ extends continuously up to the open arcs $C_1$ and $C_2$,

\item   the restrictions of the map 
 $u_n$ to $C_1$ and $C_2$ have both a finite limit at $x_0$,

\item there exists a fixed open vertical segment 
$\gamma := {x_0}\times (a,b) \subset \E(\kappa,\tau)$ such that 
$\gamma \subset \partial \Sigma_n$ and $\gamma$ is independant of $n$.
\end{itemize}
Then  the Gaussian curvature $K_n$ of the surfaces $S_n$ is uniformly bounded 
in the neighborhood of each point of $\gamma$. 
Namely, for any $p\in \gamma$ there exist  $R_p, K_p >0$, 
 satisfying for any $n$
\begin{equation*}
d_n (p,\partial \Sigma_n \setminus \gamma) > 2R_p\, ; \quad \text{and} \quad
 \lvert K_n (x) \lvert \leq K_p
\end{equation*}
 for any 
$x\in \Sigma_n $ such that $d_n(p,x) < R_p$, where $d_n$ denotes the  intrinsic 
distance   on $\Sigma_n$.
\end{proposition}

\vskip1mm

\noindent {\em Outline of the proof.} 
 We first choose points $p_i \in C_i$, $i=1,2$, and we denote by 
  $\wt C_i $ the open arc of $C_i$  with endpoints $p_i$ and $x_0$, $i=1,2$.
 Then we choose  an  open arc 
 $\wt C_3 \subset \Omega$  with endpoints $p_1$ and $p_2$ such that the bounded domain 
 $\wt \Omega \subset \Omega$ with boundary 
 $\wt C_1 \cup \wt C_2 \cup \wt C_3 \cup \{ x_0, p_1, p_2\}$, 
 is convex.
 
 Let $\wt \Gamma_n \subset \E(\kappa,\tau)$ be the Jordan curve constituted with the graph of 
 $u_n$ over $\wt C_i$, $i=1,2,3$, the points $(p_j, u_n (p_j))$, $j=1,2$, and a vertical 
closed segment 
 $\wt \gamma_n$ above $x_0$. Thus $\gamma\subset \wt\gamma_n$ for any $n$.
 
 Since $\wt \Omega$ is convex  we deduce from 
Proposition \ref{P.embedded} in the Appendix,
 that  
 $\wt\Sigma_n := \Sigma_n \cap \ov{\wt\Omega\times \R}$ is an area minimizing disc, 
 solution of the 
 Plateau problem for the Jordan curve $\wt \Gamma_n$. Thus we can  apply the reflection 
principle around  
$\gamma$ to $\wt\Sigma_n$, see \cite[Proposition 3.4]{SE-T9}. Then we proceed as in the proof of 
Proposition \ref{P.courbure-fibre}.

\vskip1mm

Observe that the same result holds also for $\sd \times \R$.

\bigskip 

\section{Appendix}

Let $\Omega\subset \M(\kappa)$, $\kappa \leq 0$, be a bounded convex domain 
 bounded by a  $C^0$ Jordan curve $\Gamma:= \partial \Omega $, 
and let $f : \Gamma \longrightarrow \R$ be a piecewise 
continuous function, allowing a finite number of discontinuities. 

We denote by $\wt \Gamma \subset \E (\kappa, \tau) $, $\tau \geq 0$, the graph of $f$. 
Namely, if $f$ is continuous then $\wt \Gamma$ is a Jordan curve with a one-to-one projection on 
$\Gamma$.
If $f$ 
has discontinuity points then $\wt \Gamma$ is constituted of a finite number of simple arcs 
admitting a one-to-one vertical projection  on some subarc of $\Gamma$, and a vertical segment over 
each point of $\Gamma$ where $f$ is not continuous.

 We consider also a $C^0$ bounded convex domain $\Omega$ in $\sd$. In this case 
$\wt \Gamma\subset \sd \times \R$.

\begin{proposition}\label{P.embedded}
 There exists an embedded area minimizing disc $\Sigma$ in 
 $\ov \Omega \times \R \subset\E (\kappa, \tau)$,  $\kappa \leq 0$, $\tau \geq 0$, 
 $($or in $\sd \times \R)$ with boundary $\wt \Gamma$. 
 Furthermore
\begin{itemize}
 \item $ {\rm int}(\Sigma)= \Sigma \setminus \wt \Gamma$ is a vertical graph over $\Omega$,

 \item  If $\Sigma_0\subset\E (\kappa, \tau)$  $($or in $\sd \times \R)$ is any 
minimal surface bounded by $\wt \Gamma$ 
such that 
 $\Sigma_0\setminus \wt \Gamma$ is a vertical graph over $\Omega$, then $\Sigma_0=\Sigma$.
 
\end{itemize}
%
\end{proposition}

\begin{remark}
 For the existence  of an embedded minimal disc in $\ov \Omega \times \R$ bounded by 
 $\wt \Gamma$, we cannot use the well-known result of Meeks-Yau \cite[Theorem 1]{Meeks-Yau}, since 
the boundary of  $\ov \Omega \times \R$ does not have the required regularity.
\end{remark}

\begin{proof}
 We perform the proof in $\E(\kappa,\tau)$, the proof in $\sd \times \R$ is 
analougous.

Assume first that  $f$ has no discontinuity points. Since $\E (\kappa, \tau) $ is 
homogeneous, we deduce from a result of Morrey \cite{Morrey} that  there exists a minimizing area 
disc $\Sigma$ bounded by $\wt \Gamma$. Since $\Omega$ is  a convex domain, we deduce from the 
maximum 
principle  that ${\rm int}(\Sigma)\subset \Omega \times \R$ (we use also the fact that if  
$\gamma\subset \M(\kappa)$ is a geodesic line, then $\gamma \times \R$ is a minimal 
surface of 
$\E (\kappa, \tau)$). 

Furthermore as  $\wt \Gamma$ has a 
one-to-one vertical projection on $\Gamma$, we infer that ${\rm int}(\Sigma)$ is a vertical graph 
over $\Omega$, as in Rado's theorem \cite[Theorem 16]{Lawson}.

Thus $\Sigma$ is an embedded area minimizing disc bounded by $\wt \Gamma$ and is a 
vertical graph over $\ov \Omega$.
 Clearly, $\Sigma$ is the unique minimal graph bounded by $\wt \Gamma$. The same 
affirmation holds in $\sd \times \R$.

\smallskip

Assume now that $f$ has a finite number of discontinuity points, let 
$x_0\in \Gamma$ be such a point.  Giving an orientation to $\Gamma$, we have 
therefore
\begin{equation*}
 \lim_{\substack{x \to x_0\\ x<x_0}} f(x) \not= \lim_{\substack{x \to x_0\\ x > x_0}} f(x).
\end{equation*}
Let $(p_n)$ be a sequence on $\Gamma$ such that 
\begin{itemize}
 \item $p_n \to x_0$ and $p_n > x_0$,
 \item for any $n$, the function $f$ is continuous on the arc 
 $(x_0, p_n]$ of $\Gamma$.
\end{itemize}
Now we modify $f$ on the closed arc $[x_0, p_n]$ in such a way that the new function $f_n$ is 
strictly monotonous on $[x_0, p_n]$ and satisfies
 
\begin{equation*}
 f_n(x_0)=  \lim_{\substack{x \to x_0\\ x<x_0}} f(x) , \qquad 
 f_n(p_n) = f(p_n)\qquad \text{and} \qquad f_n(x)\not= f(x)
\end{equation*}
for any $x\in (x_0,p_n)$.

We assume that we have modified in the same way $f$ in a neighborhood of each point of 
discontinuity and we continue denoting $f_n$ the new function. Clearly, we have 
\begin{itemize}
 \item for each point $x\in \Gamma$ where $f$ is continuous: $f_n(x) \to f(x)$,
 \item the function $f_n$ is continuous on $\Gamma$.
\end{itemize}
We denote by $\wt \Gamma_n$ the graph of $f_n$.

We know from the beginning of the proof, that for any $n$ there exists an (unique) embedded area 
minimizing 
disc $\Sigma_n$ bounded by the Jordan curve $\wt \Gamma_n$, which is a vertical graph over 
$\ov \Omega$. Let $ u_n : \ov \Omega \longrightarrow \R$ be the function whose the graph is 
$\Sigma_n$. 

Since the sequence $(f_n)$ is uniformly bounded on $\Gamma$, we deduce from the maximum principle 
that the sequence $(u_n)$ is uniformly bounded on $\ov \Omega$ too. 
In \cite[Section 4.1]{N-E-T} it is stated a compactness principle in the case where the ambient 
space is Heisenberg space, but it can be stated and proved in the same way in $\E(\kappa,\tau)$, 
 $\kappa \leq 0$, $\tau \geq 0$, and also in $\sd \times \R$. 

Therefore, up to considering a subsequence, 
we can assume that the sequence of restricted maps
$(u_{n|\Omega})$ converges on $\Omega$, for the topology $C^2$ and uniformly on any compact subset 
of $\Omega$, to a function 
$u : \Omega \longrightarrow \R$ satisfying the minimal surface equation. Thus the graph of $u$ is a 
minimal surface $\mathring \Sigma$ which has a one-to-one projection on $\Omega$.

\smallskip

Now we recall a result in  $\E(\kappa,\tau)$ for $\kappa < 0$ and $\tau >0$  proved in 
\cite[Lemma 5.3]{Younes}, 
but which holds more generally for $\kappa \leq 0$ and $\tau \geq 0$, and also in $\sd \times \R$.

Let $T\subset \M(\kappa)$ be an isosceles geodesic triangle and let 
$U\subset \M(\kappa)$ be the bounded convex domain with boundary $T$. We denote by $A,B,C$ the open 
sides of $T$ and by $a,b,c$ the vertices of $T$. Assume that $A$ and $B$ have same length 
and that $c$ is the common endpoint of $A$ and $B$.

Then, for any real number $\alpha$ there exists a continuous function 
$v : U\cup (T\setminus\{a,b\})  \longrightarrow [0, \alpha] \subset \R$ such that
\begin{enumerate}
 \item $v$ is $C^2$ on $U$ and satisfies the minimal surface equation,
 
 \item $v=0$ on $A\cup B\cup \{c\}$, and $v=\alpha$ on $C$.
\end{enumerate}

\smallskip

Using those surfaces as barrier at each point  $x\in \Gamma$ where $f$ is continuous, as in the 
proof of \cite[Theorem 3.4]{ST1}, we infer that $u$ extends continuously up to $x$ setting 
$u(x)=f(x)$. We deduce that the topological boundary of $\mathring \Sigma$ is the Jordan curve $\wt 
\Gamma$.

Setting $\Sigma:= \mathring \Sigma \cup \wt \Gamma$, we have 
\begin{itemize}
 \item $\Sigma$ is an embedded minimal disc in $\ov \Omega \times \R$,
 
 \item $\partial \Sigma = \wt \Gamma$,
 
 \item ${\rm int}(\Sigma)= \mathring \Sigma \subset \Omega \times \R$, and 
 ${\rm int}(\Sigma)$ is a vertical graph over $\Omega$.
\end{itemize}

Now we want to prove that $\Sigma$ is an area minimizing disc.

\smallskip

Observe first that $(Area (\Sigma_n))$ is a bounded sequence. Thus, up to considering a 
subsequence, we 
can assume that $(Area (\Sigma_n))$ is a convergent sequence. Recall that the sequence $(u_n)$ 
converges for the $C^2$ topology to $u$, uniformly on any compact subset of $\Omega$. Therefore for 
any compact subset $K\subset \Omega$ we have 
\begin{equation*}
 Area (graph (u_{|K}))= \lim Area (graph (u_{n|K}))
 \leq \lim (Area (\Sigma_n)).
\end{equation*}
 Thus we get 
\begin{equation*}
 Area (\Sigma) \leq \lim (Area (\Sigma_n)).
\end{equation*}

Suppose by contradiction that $\Sigma$ is not a minimizing area disc.
Thus, considering a solution of the Plateau problem for 
$\wt \Gamma$, 
there exists a minimal immersed disc $S\subset \E(\kappa,\tau)$ 
  (or in $\sd \times \R$), with 
\begin{itemize}
 \item $\partial S =\wt \Gamma$,
 
\item $Area (S) < Area (\Sigma)$. 
\end{itemize}
Using the maximum principle we get
\begin{itemize}
 \item $S \subset \ov \Omega \times \R$,
 
 \item $S\setminus \wt \Gamma \subset \Omega \times \R$.
\end{itemize}

Let again $x_0\in \Gamma$ be a point where $f$ is not continuous, we use the same notations as in 
the beginning of the proof. 
For  any $n$ we denote by $s_n$ the bounded subset of the cylinder 
$\Gamma \times \R$ bounded by the following curves:
\begin{itemize}
 \item the graph of $f_n$ over the arc $[x_0, p_n]$ of $\Gamma$, 
 
 \item the graph of $f$ over the arc $(x_0, p_n]$,
 
 \item the vertical closed segment above $x_0$, with endpoints \newline
 $(x_0,\lim_{\substack{x \to x_0\\ x<x_0}} f(x))$ and 
  $(x_0,\lim_{\substack{x \to x_0\\ x>x_0}} f(x))$. 
\end{itemize}

Clearly we have $Area (s_n) \to 0$. Furthermore, by  the foregoing discussion there 
exists $\varepsilon >0$ 
such that $Area (S) < Area (\Sigma_n) -\varepsilon$ for any $n$ large enough.

 Let $\{x_i\}\subset \Gamma$ be the finite set of points where $f$ is not 
continuous. For any point $x_i$ we denote by $s_n (x_i)$ the piece of $\Gamma \times \R$ 
constructed as above, 
corresponding to $x_i$. 

Observe now that $S_n := S \cup \left( \cup_i s_n (x_i)\right) $ is a disk with boundary the Jordan 
curve $\wt \Gamma_n$. By construction we have $Area (S_n) < Area (\Sigma_n)$ for any $n$ large 
enough, 
contradicting the fact that $\Sigma_n$ is a minimizing area disc with boundary $\wt \Gamma_n$.

\smallskip

Thus $\Sigma$ is an embedded minimizing area disc with boundary 
$\wt \Gamma$ such that $\Sigma\setminus \wt \Gamma$ is a vertical graph over $\Omega$.

\smallskip

It remains to prove that 
if $\Sigma_0\subset\E (\kappa, \tau)$ (or in $\sd \times \R$), is a 
minimal surface  bounded by 
$\wt \Gamma$ such that $\Sigma_0\setminus \wt \Gamma$  is a vertical graph over $\Omega$, then 
$\Sigma_0=\Sigma$.

Let $u_0: \Omega\longrightarrow \R$ be the function whose the graph is 
$\Sigma_0\setminus \wt \Gamma$. In \cite[Theorem 1.3]{Pinheiro} A.L. Pinheiro 
proves a general maximum principle in a Riemannian product $\M \times \R$, where 
$\M$ is Riemannian surface. The proof can be adapted to $\E(\kappa, \tau)$, $\kappa \leq 0$, 
$\tau \geq 0$, thus $u_0=u$, that is $\Sigma_0=\Sigma$.

Observe also that in $\PSL (2,\R)$, R. Younes proved directly that $u_0=u$, see  
\cite[Theorem 1.1]{Younes}. The proof can be adapted in $\E(\kappa, \tau)$, $\kappa \leq 0$, 
$\tau \geq 0$, and also in $\sd \times \R$.

Thus $\Sigma_0=\Sigma$, which concludes the proof.
\end{proof}

\bigskip 

\bigskip

\begin{proposition}\label{P.Line}
 Let $M\subset \R^3$ be a complete minimal surface containing a segment 
 of a straight line $D$. Then the whole line $D$ belongs to $M$: $D\subset M$.
\end{proposition}

\begin{proof}
We denote by $x, y, z$  the coordinates on $\R^3$. Up to an isometry of $\R^3$ we can 
assume that $D$ is the $x$-axis: 
$D=\{(x,0,0),\ x \in \R \}$. 

 By assumption there exist real numbers $a<b$ such that 
$(x,0,0)\in M$ for any $x\in [a,b]$.
 
We set 
\begin{align*}
 B&:=\sup \{t >a,\ (x,0,0)\in M \ \text{for any } x\in [a,t]\},\\
 A&:=\inf \{t<b,\ (x,0,0)\in M \ \text{for any } x\in  [t,b]   \}.
\end{align*}
We are going to prove that $A=-\infty$ and $B=+\infty$ to conclude that $D\subset M$.

\smallskip 

We have $B\geq b$. 
Assume by contradiction that $B\not=+\infty$. Since $M$ is a complete surface we have  
$(B,0,0) \in M$. Let 
$P\subset \R^3$ 
be the plane containing $D$ and the orthogonal direction of $M$ at $(B,0,0)$. 

Since the surfaces $M$ and $P$ are transverse at $(B,0,0)$, their intersection in a 
neighborhood of $(B,0,0)$ is an analytic arc $\gamma$. Furthermore, up to choosing a 
smaller arc, we can assume that $\gamma$ is the graph,  in $P$, of an analytic 
function $f$ over 
the interval $[B-\varepsilon, B+\varepsilon]$ for $\varepsilon>0$ small enough. 
Since $f$ is an analytic function satisfying $f(x)=0$ for any $x\in [B-\varepsilon, B]$, 
we deduce that $f(x)=0$ for any $x\in [B-\varepsilon, B+\varepsilon]$. Therefore we have 
$(x,0,0) \in M$ for any $x\in [a,B+\varepsilon]$, contradicting the definition of $B$.

Thus we have $B=+\infty$. We prove in the same way that $A=-\infty$, concluding the proof.
\end{proof}

\end{document}